\documentclass[12pt]{amsart}

\usepackage{amsmath}
\usepackage{amsfonts}
\usepackage{amssymb}
\usepackage{amscd}
\usepackage{graphicx}
\RequirePackage[dvipsnames,usenames]{color}
\usepackage{soul,xcolor}
\setstcolor{red}
\usepackage{stmaryrd}
\SetSymbolFont{stmry}{bold}{U}{stmry}{m}{n}
\usepackage{mathtools}
\usepackage{booktabs}
\usepackage{multirow}\usepackage{subcaption}

\usepackage{mathdots}
\newtagform{tiny}{\tiny(}{)}

\usepackage{mathtools}
\usepackage{hyperref}
\usepackage[margin=1.25in]{geometry}

\usepackage{amsthm}
\usepackage{comment}
\usepackage[all,cmtip]{xy}
\usepackage{tikz-cd}
\usetikzlibrary{cd}

\usepackage[all]{xy}

\usepackage{etoolbox}
\apptocmd{\sloppy}{\hbadness 10000\relax}{}{}

\usepackage{enumitem}

\title{Frobenius liftable hypersurfaces}
\author{Tatsuro Kawakami}
\address{Graduate School of Mathematical Sciences, University of Tokyo, 3-8-1 Komaba,
Meguro-ku, Tokyo 153-8914, Japan}
\email{tatsurokawakami0@gmail.com}

\author{Supravat Sarkar} 
\address{Department of Mathematics \\
Fine Hall, Washington Road\\
Princeton University\\  
Princeton NJ 08544-1000, USA}
\email{ss6663@princeton.edu}

\author{Jakub Witaszek} 
\address{Department of Mathematics \\
Fine Hall, Washington Road\\
Princeton University\\  
Princeton NJ 08544-1000, USA}
\email{jwitaszek@princeton.edu}

\def\phi{\varphi}
\def\epsilon{\varepsilon}
\def\tilde{\widetilde}

\def\log{\operatorname{log}}

\def\Spec{\operatorname{Spec}}

\def\codim{\operatorname{codim}}
\def\Pic{\operatorname{Pic}}

\def\Gr{\operatorname{Gr}}

\newcommand{\Q}{\mathbb{Q}}

\newcommand{\Z}{\mathbb{Z}}

\newcommand{\sO}{\mathcal{O}}
\newcommand{\sHom}{\mathcal{H}\! \mathit{om}}

\newcommand{\stacksproj}[1]{{\cite[\href{https://stacks.math.columbia.edu/tag/#1}{Tag~{#1}}]{stacks-project}}}

\newsavebox{\pullback}
\sbox\pullback{%
\begin{tikzpicture}%
\draw (0,0) -- (1ex,0ex);%
\draw (1ex,0ex) -- (1ex,1ex);%
\end{tikzpicture}}

\newsavebox{\pullbackdl}
\sbox\pullbackdl{%
\begin{tikzpicture}%
\draw (-1ex,0ex) -- (0ex,0ex);%
\draw (0ex,-1ex) -- (0ex,0ex);%
\end{tikzpicture}}

\newsavebox{\pushoutdr}
\sbox\pushoutdr{%
\begin{tikzpicture}%
\draw (-1ex,-1ex) -- (-1ex,0ex);%
\draw (-1ex,0ex) -- (0ex,0ex);%
\end{tikzpicture}}

\theoremstyle{plain}
\newtheorem{thm}{Theorem}[section] 

\newtheorem{prop}[thm]{Proposition}

\newtheorem{lem}[thm]{Lemma}
\theoremstyle{definition} 
\newtheorem{defn}[thm]{Definition}

\newtheorem{setting}[thm]{Setting}

\theoremstyle{remark}
\newtheorem{rem}[thm]{Remark}

\newtheorem{defn and notation}[thm]{Definition and Notation}

\newtheorem*{cl}{Claim}

\theoremstyle{plain}
\newtheorem{theo}{Theorem}

\numberwithin{equation}{thm}
\keywords{Frobenius liftable; Toric varieties}
\subjclass[2020]{13A35,14M25}

\begin{document}
\tolerance = 9999

\begin{abstract}
Let $D$ be a reduced divisor in $\mathbb P^n_k$ for an algebraically closed field $k$ of positive characteristic $p > 0$. We prove that if $(\mathbb P^n_k, D)$ is Frobenius liftable modulo $p^2$, then $D$ is a toric divisor. As a corollary, we show that if there exists a finite surjective morphism $f\colon Y\to X$ onto  a smooth projective complex variety $X$ of Picard rank $1$
    such that $(Y, f^{-1}(D)_{\mathrm{red}})$ is a toric pair, 
    then $X$ is the projective space and $D$ is a toric divisor. 
\end{abstract}

\maketitle
\markboth{Tatsuro Kawakami, Supravat Sarkar, and Jakub Witaszek}{FROBENIUS LIFTABLE HYPERSURFACES}


\section{Introduction}

The study of endomorphisms of algebraic varieties is an interesting problem at the intersection of algebraic geometry and dynamics. Our article is motivated by the following well-known conjecture on totally invariant divisors in $\mathbb P^n_{\mathbb C}$. Let $D$ be a prime divisor on $\mathbb P^n_{\mathbb C}$ and let $f \colon \mathbb P^n_{\mathbb C} \to \mathbb P^n_{\mathbb C}$ be an endomorphism of degree $d \geq 2$ such that  $f^{-1}(D)_{\rm red} = D$. Then $D$ is conjectured to be linear (see \cite{Horing17, Mabed23, Beauville01} for some partial results). 

In positive characteristic, an analogue of a polarised endomorphism is given by a lift of Frobenius as discussed in \cite{AWZ,AWZ2}. The following is the key result of this article. We refer to Definition \ref{def:F-lift} for the definition of F-liftability for pairs. 
\begin{theo}\label{Introthm:main}
    Let $k$ be an algebraically closed field of characteristic $p>0$ and
    let $D$ be a reduced divisor on $\mathbb{P}_k^n$.
  If $(\mathbb{P}^n_k, D)$ is $F$-liftable, then $D$ is a toric divisor with respect to the standard toric structure on $\mathbb{P}^n_k$, up to an automorphism of $\mathbb{P}^n_k$.
\end{theo} 
Note that Theorem \ref{Introthm:main} implies the original conjecture on totally invariant divisors on $\mathbb P^n_{\mathbb C}$ when $f$ is an unramified lift of Frobenius, thus providing more evidence towards its validity.

As discussed in \cite{AWZ}, the study of Frobenius liftability is also related to the Occhetta-Wi\'{s}niewski conjecture predicting that every smooth complex projective variety which is an image of a projective toric variety is automatically toric. This conjecture is known when $\rho(X)=1$ by \cite{Occhetta-Wisniewski}. As a corollary to Theorem \ref{Introthm:main}, we can obtain the following stronger result. 
\begin{theo}\label{Introthm:main2}
    Let $X$ be a smooth projective variety of Picard rank $1$ with $\dim X=n$ over an algebraically closed field $k$ of characteristic zero, and
    let $D$ be a reduced divisor on $X$.
    Suppose there exists a finite surjective morphism $f\colon Y\to X$
    such that $(Y, f^{-1}(D)_{\mathrm{red}})$ is a toric pair.
    Then $X\cong \mathbb{P}^n_k$, and $D$ is a toric divisor with respect to the standard toric structure on $\mathbb{P}^n_k$, up to an automorphism of $\mathbb{P}^n_k$.
\end{theo} 

\subsection{Idea of the proof of Theorem \ref{Introthm:main}} 
For simplicity, assume that $D$ is irreducible. 
Set $X \coloneqq \mathbb{P}^n_k$. 
Firstly, one can show that $H^0(X, \Omega^{[1]}_X(\log D))=0$ for any prime divisor $D$ as a consequence of log Bott vanishing for $(X,D)$ (see Theorem \ref{thm:global section of F-lift pair}).

Now, suppose by contradiction that $\deg D \geq 2$ but $(X, D)$ is Frobenius liftable. An easy calculation shows that in this case 
\[
H^0(L, \Omega^{[1]}_X(\log D)|_L) \neq 0
\]
for a general line $L$ (Lemma \ref{lem:general line}). Pick a section $s$ of $H^0(L, \Omega^{[1]}_X(\log D)|_L)$. To obtain a contradiction, we would like to deform $L$ together with $s$ to construct a global section of $\Omega^{[1]}_X(\log D)$. At first glance, this seems impossible as $s$ will not deform uniquely with $L$, and so there is no way to patch together sections over different lines. To address this, we use a map 
\[
\xi \colon F^*\Omega^{[1]}_X(\log D) \to \Omega^{[1]}_X(\log D)
\]
induced by the Frobenius lift (see (\ref{eq:xi-Frobenius-lift})) and the $\mathbb F_p$-constructible  \'etale subsheaf  $\Omega^{[1]}_X(\log D)^{\xi}$ of $\Omega^{[1]}_X(\log D)$ consisting of sections fixed by $\xi$ (Lemma \ref{lemma:magic-cover}). This subsheaf is introduced in \cite{AWZ}, where it is called the \emph{magic cover}.

We may assume that $s$ is also fixed by $\xi$ (Lemma \ref{lem:p-linear}). As we deform $L$, the section $s$ will deform uniquely as an element of $\Omega^{[1]}_X(\log D)^{\xi}$. Since $\mathbb P^n$ is simply connected, these deformations must patch together to give a non-zero section of $\Omega^{[1]}_X(\log D)^{\xi}$ yielding a contradiction.

\subsection*{Acknowledgements}  
We would like to thank Piotr Achinger, Mircea Musta{\c{t}}{\u{a}}, Maciej Zdanowicz, Mihnea Popa, and Shou Yoshikawa for valuable conversations related to the content of this article. 

\section{Preliminaries}
\subsection{Notation and terminology} \label{ss:notation}
Throughout this paper, we use the following notations:
\begin{itemize}
\item A \textit{variety} is an integral separated scheme of finite type over a field. A curve is a variety of dimension one, while a surface is a variety of dimension two. 

\item Given an integral scheme $X$ of finite type over $k$ and a reduced divisor $D$ on $X$, we say $(X,D)$ is \textit{normal crossing} at $x\in X$ if there exist an open neighborhood $U$ of $x\in X$ and an \'etale map $h\colon U\to \mathbb{A}^{\dim\,X}_k$ such that $D=h^*(\{x_1x_2\cdots x_m=0\})$ for some $m\geq 0$. We say $(X,D)$ is \emph{normal crossing at codimension $k$ points} of $X$ for an integer $k$ if $(X,D)$ is normal crossing at every point of codimension $k$. We simply say that $(X,D)$ is normal crossing if it is normal crossing at every point.

\item Given an integral scheme $X$ of finite type over $k$ and a reduced divisor $D$ on $X$, we say that $(X,D)$ is \textit{simple normal crossing} (\textit{snc} for short) if $(X,D)$ is normal crossing and each irreducible component of $D$ is smooth.

\item Given a normal crossing pair $(X,D)$ one can define a locally free sheaf $\Omega^i_X(\log D)$ which is stable under \'etale base change (see \cite[4.0]{Katz70}). When $(X,D)$ is simple normal crossing with $D = \sum_{i=1}^r D_i$ for distinct prime divisors $D_i$, we have an exact sequence \cite[Section 2.3]{AWZ}:
\begin{equation} \label{eq:ses-diff-snc}
0 \to \Omega^1_X \to \Omega^1_X(\log D) \to \bigoplus_{i=1}^r \sO_{D_i} \to 0.
\end{equation}

\item Given a normal variety $X$ over a perfect field and a reduced divisor $D$ on $X$, we denote $j_{*}(\Omega_{U}^{i}(\log D|_U))$ by $\Omega_X^{[i]}(\log D)$, where $j\colon U\hookrightarrow X$ is the inclusion from the normal crossing locus $U$ of $(X, D)$, which is the largest open subvariety $U$ of $X$ such that $(U,D|_U)$ is normal crossing.

\item Given a smooth variety $X$, we have a group homomorphism 
\begin{equation} \label{eq:dlog}
\mathrm{dlog}\colon \Pic(X)\to H^1(X,\Omega^1_X)
\end{equation}
defined as follows: if $\{f_{ij}\}$ are the transition functions of $L$, then $\mathrm{dlog}(L)=\{df_{ij}/f_{ij}\}\in H^1(X,\Omega^1_X)$ (cf.\ \cite[Exercise III.7.4(c)]{hartshorne2013algebraic}). In what follows, we will  denote $\mathrm{dlog}(L)$ by $c(L)$. Note that if $X$ is of characteristic $p>0$, then we have an induced map $\mathrm{dlog} \colon \Pic(X) \otimes_{\mathbb Z} \mathbb F_p \to H^1(X,\Omega^1_X)$. 

\item For the definition of the singularities of pairs appearing in the MMP (such as \emph{klt, plt, lc}), we refer to \cite[Definition 2.8]{Kol13}. Note that we always assume that the boundary divisor is effective although \cite[Definition 2.8]{Kol13} does not impose this assumption.
\item We say that a normal variety $X$ over a perfect field $k$ of characteristic $p>0$ is \emph{globally $F$-split} if Frobenius $F^* \colon \sO_X \to F_*\sO_X$ splits in the category of sheaves of $\sO_X$-modules. We say that $X$ is $F$-pure if $F^* \colon \sO_{X,x} \to F_*\sO_{X,x}$ splits for localizations $\sO_{X,x}$ at each point $x \in X$.
\item Let $k$ be a field of positive characteristic. We say that a map $\varphi \colon V \to W$ between $k$-vector spaces is \emph{Frobenius-linear} if
\[
\varphi(v_1 + v_2) = \varphi(v_1)+\varphi(v_2) \quad \text{ and } \quad \varphi(av) = a^p\phi(v)
\]
for $v_1,v_2, v \in V$ and $a \in k$. This is equivalent to saying that the induced map $\varphi \colon F^*V \to W$  is linear.
\end{itemize}

We shall use the following fact about Frobenius-linear maps.

\begin{lem} \label{lem:p-linear}
Let $k$ be an algebraically closed field, let $V$ be a $k$-vector space of finite dimension $r$,  and let $\varphi \colon V \to V$ be a Frobenius-linear endomorphism. Assume that $\varphi$ is injective. Then it is bijective and the locus of $\varphi$-fixed points admits an isomorphism:
\[
V^{\varphi} \cong \mathbb F_p^r
\]
of $\mathbb F_p$-vector spaces.
\end{lem}
\begin{proof}
This follows by the proof of \stacksproj{0A3L}.
\end{proof}
\subsection{Frobenius liftable pairs}

We recall the definition of a Frobenius liftable pair. Informally speaking, a pair $(X,D)$ is Frobenius liftable mod $p^2$ if it admits a lift $(\widetilde{X}, \widetilde{D})$ over $W_2(k)$ and there exists an endomorphism $\widetilde{F}$ lifting Frobenius for which each irreducible component of $\widetilde{D}$ is totally invariant.
\begin{defn}\label{def:F-lift}
Let $k$ be an algebraically closed field of characteristic $p>0$.
Let $X$ be a normal variety over $k$ and $D=\sum_{r=1}^n D_r$ a reduced divisor on $X$, where each $D_r$ is an irreducible component.  
We denote the ring of Witt vectors of $k$ of length two by $W_2(k)$. We say that $(X,D)$ is \textit{$F$-liftable}
if there exist 
\begin{itemize}
	\item a flat morphism $\widetilde{X} \to \Spec W_2(k)$  with a closed immersion $i\colon X\hookrightarrow \widetilde{X}$ over the natural map of bases $\Spec k \hookrightarrow \Spec W_2(k)$, 
	 \item closed subschemes $\widetilde{D}_r$ of $\widetilde{X}$ flat over $W_2(k)$ for all $r\in\{1,\ldots,n\}$, and 
	 \item a morphism $\widetilde{F}\colon \widetilde{X}\to \widetilde{X}$ over the Frobenius of the base $W_2(k)$
\end{itemize}
	such that 
\begin{itemize}
	\item the induced morphism $i\times_{W_2(k)}k \colon X\to \widetilde{X}\times_{W_2(k)} k$ is an isomorphism, 
	\item $(i\times_{W_2(k)}k)(D_r)= \widetilde{D}_r\times_{W_2(k)} k$ for all $r\in\{1,\ldots,n\}$, 
    \item $\widetilde{F}\circ i=i \circ F$, where $F$ is the Frobenius on $X$,
    \item $\widetilde{D}|_{\tilde{U}}$ is normal crossing relative to $W_2(k)$,
    where $\tilde{D}=\sum_{r=1}^n \widetilde{D}_r$ and $\tilde{U}\subset \tilde{X}$ is the lift of the normal crossing locus of $(X,D)$, and
    \item $\widetilde{F}^{*}(\widetilde{D}_r|_{\tilde{U}})=p(\widetilde{D}_r|_{\tilde{U}})$ for all $r\in\{1,\ldots,n\}$.
\end{itemize}
We say that $(X,D)$ is \textit{$F$-liftable at a closed point $P\in X$} if there exists an open neighborhood $U$ of $P\in X$ such that $(U,D|_{U})$ is $F$-liftable.
We say that $(X,D)$ is \textit{locally $F$-liftable} if the pair $(X,D)$ is $F$-liftable at every closed point $P\in X$.
\end{defn}

Note that the above definition is a bit more restrictive than the one in \cite[Definition 2.2]{Kawakami-Takamatsu} as the latter does not assume that the normal crossing locus lifts to a normal crossing locus. This additional assumptions is necessary for applying deformation theoretic arguments as in \cite{AWZ}.

\begin{rem}\label{rem:Frobenius liftable singularities}
With notation as in Definition \ref{def:F-lift}, the following statements hold.
\begin{enumerate}
    \item If $(X,D)$ is $F$-liftable (resp.~locally $F$-liftable), then it is globally $F$-split (resp.~$F$-pure) (the local case is \cite[Corollary 3.7]{Kawakami-Takamatsu}; the proof of the global case is completely analogous).
    \item If $(X,D)$ is $F$-pure and $K_X+D$ is $\Q$-Cartier, then $(X,D)$ is lc \cite[Theorem 3.3]{Hara-Watanabe} and
    it is normal crossing at codimension $\leq 2$ points of $X$ \cite[Corollary 2.32]{Kol13}. 
\end{enumerate}
\end{rem}

Finally, suppose that $(X,D)$ is an $F$-liftable pair. Then there exists a morphism
\begin{equation} \label{eq:xi-Frobenius-lift}
\xi \colon F^*\Omega^{[1]}_X(\log D) \to \Omega^{[1]}_X(\log D).
\end{equation}
which is an isomorphism at the generic point of $X$. Indeed, such a morphism exists on a normal crossing locus (see \cite[(3.1)]{AWZ2} or \cite[\S 4.2]{DI87}) from which we can reflexify it to the whole $X$.

Next, we look at the (\'etale) subsheaf $(\Omega^1_X(\log D))^{\xi}$ of $\xi$-invariant forms.

\begin{lem}[{cf.\ Lemma \ref{lem:p-linear}}] \label{lemma:magic-cover}
With notation as in Definition \ref{def:F-lift}, suppose that $(X,D)$ is a normal crossing pair. Then the \'etale subsheaf $(\Omega^1_X(\log D))^{\xi}$ of $\Omega^1_X(\log D)$ is a constructible sheaf of $\mathbb F_p$-vector spaces. 

Moreover, over the open dense subset $U \subseteq X$ where $\xi$ is an isomorphism, $(\Omega^1_X(\log D))^{\xi}$ is an $\mathbb F_p$-locally constant sheaf of rank $p^{\dim X}$.
\end{lem}
\begin{proof}
The proof is completely analogous to that in \cite[Lemma 3.2.7]{AWZ} with $\Omega^1_X$ replaced by $\Omega^1_X(\log D)$.
\end{proof}

\subsection{Toric images}
Last, we recall the following theorem of Occhetta-Wi\'{s}niewski and its positive characteristic variant from \cite[Theorem 3.9]{Kawakami-Totaro}.

\begin{thm}[\textup{\cite[Theorem 1.1]{Occhetta-Wisniewski} (cf.~\cite[Theorem 3.9]{Kawakami-Totaro})}]\label{thm:Occhetta-Wisniewski}
    Let $X$ be a smooth projective variety of Picard
    number 1 over an algebraically closed field.
    Let $Y$ be a proper toric variety. If there exists
    a separable surjective morphism $Y\to X$,
    then $X$ is isomorphic to projective space.
\end{thm}

\section{Proof of main theorems}

\begin{setting}\label{setting:global 1-forms}
     Let $k$ be an algebraically closed field.
     Let $X$ be a smooth variety over $k$ of dimension $n\geq 2$, and let $D$ be a prime divisor on $X$ such that $(X,D)$ is normal crossing at points of codimension $\leq 2$.
     Let $j\colon U\hookrightarrow X$ be the inclusion of the normal crossing locus of $(X,D)$, and $Z\coloneqq X\setminus U$.  
 \end{setting}

\begin{lem}[{cf.\ \cite[equation (2.6) and Corollary 2.8]{dolgachev2007logarithmic}}]\label{lem:exact sequence}
    In the situation of Setting \ref{setting:global 1-forms},
    we have the short exact sequence
    \begin{equation}\label{residue}
    0\to \Omega^1_X\to \Omega^{[1]}_X(\log D)\to \nu_*\mathcal{O}_{D'}\to 0,
\end{equation}
    where $\nu\colon D'\to D$ is the normalization of $D$.
\end{lem}
\noindent Note that $\nu_* \sO_{D'}$ is $S_2$ as $\sO_{D'}$ is $S_2$ and $\nu$ is finite ({cf.\ \cite[Proposition 5.4]{KM98}}). 
\begin{proof}
First, we construct the map 
\[
{\rm res}|_U \colon \Omega^{[1]}_X(\log D)|_U \to \nu_*\mathcal{O}_{D'}|_U.
\]
To this end, consider a log resolution of singularities $\pi \colon \widetilde{U} \to U$ of the normal crossing pair $(U,D|_U)$ which exists, for example, by \cite{conrad2016normal}. Let $E = \pi^{-1}(D|_U)^{\rm red}$ and let $E_1, \ldots, E_r$ be the irreducible components of $E$ with $E_1$ being the normalisation of $D|
_U$. Then (\ref{eq:ses-diff-snc}) provides us with a map
\[
\Omega^1_{\widetilde U}(\log E) \to \bigoplus_{1 \leq i \leq r} \sO_{E_i},
\]
and thus with the induced projection $\Omega^1_{\widetilde U}(\log E) \to \sO_{E_1}$. By pushing forward by $\pi$, we then get a map:
\[
{\rm res}|_U \colon \Omega^{1}_U(\log D|_U) = \pi_*\Omega^1_{\widetilde U}(\log E) \to \nu_*\sO_{D'}|_U
\]
as required, where the first equality follows by the existence of a natural pullback map $\Omega^{1}_U(\log D|_U) \to \pi_*\Omega^1_{\widetilde U}(\log E)$. 

Now, consider the sequence:
\[
0\to \Omega^1_U\to \Omega^{[1]}_U(\log D|_U)\to \nu_*\mathcal{O}_{D'}|_U \to 0.
\]
We claim that it is exact. As this can be checked \'etale locally, we may assume that $(U,D|_U)$ is simple normal crossing, in which case the sequence is exact by (\ref{eq:ses-diff-snc}).

By assumptions, $Z \subseteq D$ is of codimension at least $2$ in $D$. Thus, by applying $j_*$ for $j \colon U \to X$ to the above diagram, we get an exact sequence
\[
0\to \Omega^1_X\to \Omega^{[1]}_X(\log D)\to \nu_*\mathcal{O}_{D'}.
\]
Let us denote the image of the rightmost map by $\mathcal C$. To conclude the proof of the lemma, we need to show that $\mathcal C = \nu_*\mathcal{O}_{D'}$. Since both sheaves {are torsion-free $\mathcal{O}_D$-modules and they} agree with each other outside $Z \subseteq D$ which is of codimension at least $2$ in $D$, it is enough to prove that $\mathcal C$ is $S_2$.

Since $X$ is smooth, $\Omega^1_X$ is locally-free. In particular, it is $S_3$. Moreover, $\Omega^{[1]}_X(\log D)$ is automatically $S_2$ as it is reflexive. Hence $\mathcal{C}$ is $S_2$ by depth lemma \cite[Lemma 2.60]{Kol13}. This concludes the proof of the lemma. 
\end{proof}

Although the following two results are not needed, we include them below as they provide a better understanding of the overall picture (cf.\ the proof of Theorem \ref{thm:global section of F-lift pair}).
\begin{lem}\label{boundary}
    In the situation of Setting \ref{setting:global 1-forms},
    let
    \[
    \delta_X\colon H^0(D', \mathcal{O}_{D'})(=H^0(D, \nu_*\mathcal{O}_{D'})) \to H^1(X,\Omega^1_X)
    \]
    be the connecting map induced by the short exact sequence
    \eqref{residue}.
    Then we have $\delta_X(1)=c(\mathcal{O}_X(D))$.
\end{lem}
\begin{proof}
This result is standard and follows by definition of $c(\sO_X(D))$ when $D$ is smooth (see (\ref{eq:dlog})). In general, let $S \subseteq D$ be the singular locus of $D$ and $V := X \backslash S$. Since ${\rm codim}_X(S) \geq 2$ we have that $H^1_{S}(\Omega^1_X)=0$ (see \cite[Proposition 1.2.10 (a) and (e)]{BH93}), and so the natural restriction map
    $r\colon H^1(X,\Omega^1_X)\to H^1({V},\Omega^1_{V})$ is injective.
    Now, we have the following commutative diagram
    \[
    \begin{tikzcd}
H^0(D',\mathcal{O}_{D'}) \arrow[r, "\delta_X"] \arrow[d]
& H^1(X, \Omega^1_X) \arrow[d, hook, "r" ] \\
 H^0(D'_{V},\mathcal{O}_{D'_{V}}) \arrow[r,"\delta_V" ]
& |[, rotate=0]| H^1({V}, \Omega^1_{V}),
\end{tikzcd}
\]
where {$D'_V\coloneqq \nu^{-1}(D\cap V)$} and ${\delta_V}$ is the connecting map.
Since {$D|_V$ is smooth}, we have $\delta_V(1)=c(\sO_V(D|_V))$.
Therefore, by the injectivity of $r$, we conclude.
\end{proof}

\begin{prop}\label{h00}
    In the situation of Setting \ref{setting:global 1-forms}, assume that $X$ is projective.
    If $c(\mathcal{O}_X(D))\neq 0$, then $\dim_{k}H^0(X,\Omega^{[1]}_X (\log D))=\dim_{k}H^0(X,\Omega^1_X)$.
\end{prop}
\begin{proof}
By \eqref{residue}, we have an exact sequence
   \[
   0\to H^0(X,\Omega^1_X) \to H^0(X,\Omega^{[1]}_X(\log D)) \to  H^0(D', \mathcal{O}_{D'})(\cong k)\xrightarrow{\delta_X} H^1(X,\Omega^1_X).
   \]
   Since $\delta_X(1)=c(\mathcal{O}_X(D))\neq 0$, the map $\delta_X$ is injective, and from this we can conclude.
\end{proof}

\subsection{Frobenius liftability}

The following proposition can be seen as an analogue of Totaro's result \cite[Theorem 3.1]{Totaro-endo}, which states that if $X$ admits an int-amplified endomorphism and $D$ is a totally invariant divisor, then the pair $(X, D)$ satisfies log Bott vanishing.
\begin{prop}\label{prop:log Bott}
Let $k$ be an algebraically closed field of characteristic $p>0$.
Let $(X,D)$ be a pair of a normal projective variety over $k$ of dimension $n$ and a reduced divisor $D$ such that $(X,D)$ is $F$-liftable.
Then 
\[
H^j(X,\Omega^{[i]}_X(\log D)(-D+A))=0
\]
for all ample $\Q$-Cartier $\Z$-divisors $A$ on $X$ and all $i\geq 0$ and $j>0$.
\end{prop}
\begin{proof}
    Fix $i\geq 0$.
    Let $U$ be a normal crossing locus of $(X,D)$.
    By \cite[Proposition 3.2 and Variant 3.2.2]{AWZ}, there exists a split surjection
    \[
    F_{*}\Omega^{d-i}_U(\log D)\to \Omega^{d-i}_U(\log D),
    \]
    and thus a split surjection 
    \[
    F^e_{*}\Omega^{d-i}_U(\log D)\to \Omega^{d-i}_U(\log D)
    \]
    for all $e>0$.
    By taking $\sHom_{\sO_U}(-,\omega_U)$, we have a split injection
    \[
    \Omega^{i}_U(\log D)(-D) \to F^e_{*}\Omega^{i}_U(\log D)(-D).
    \]
    Tensoring with $\sO_U(A)$, we obtain a split injection
    \[
    \Omega^{i}_U(\log D)(A-D) \to F^e_{*}\Omega^{i}_U(\log D)(p^eA-D).
    \]
    Pushing forward by the inclusion ${j \colon} U\hookrightarrow X$, we {get a split injection}
    \[
    \Omega^{[i]}_X(\log D)(A-D) \to F^e_{*}\Omega^{[i]}_X(\log D)(p^eA-D).
    \]
    {In particular, the induced map on cohomologies}
    \[
    H^j(X, \Omega^{[i]}_X(\log D)(A-D))\to H^j(X, \Omega^{[i]}_X(\log D)(p^eA-D))
    \]
    {is injective} for all $j>0$. Since $A$ is ample, we have $H^j(\Omega^{[i]}_X(\log D)(p^eA-D))$ for sufficiently large $e\gg0$ by Serre vanishing, which concludes the proof.
\end{proof}

\begin{thm}\label{thm:global section of F-lift pair}
    Let $k$ be an algebraically closed field of characteristic $p>0$.
    Let $X$ be a smooth projective variety over $k$ such that $H^1(X,\Omega^1_X)\neq 0$.
    Suppose that there exists an ample prime divisor $D$ on $X$ such that $(X,D)$ is $F$-liftable.
    Then the following statements hold:
    \begin{enumerate}
        \item[\textup{(1)}] $H^1(X,\Omega^1_X)\cong k$.
        \item[\textup{(2)}] $H^0(X, \Omega^1_X)\cong H^0(X, \Omega^{[1]}_X(\log D))$.
    \end{enumerate}
\end{thm}
\begin{rem}
    By \cite[Proposition 8.11]{Kawakami-Tanaka3} (cf.~\cite[Theorem 4.6]{vdGK03}),
the group homomorphism
\[
\mathrm{dlog}\colon \Pic(X)\otimes_{\Z}\mathbb{F}_p\to H^1(X,\Omega^1_{X}).
\]
is injective when $H^0(X,\Omega^1_X)=H^0(X,\Omega^2_X)=0$. In particular, the assumption that $H^1(X,\Omega^1_X)\neq 0$ is satisfied if $H^0(X,\Omega^1_X)=H^0(X,\Omega^2_X)=0$. 

Of course, in characteristic zero we always have $H^1(X, \Omega^1_X) \neq 0$, but it is unclear whether the same holds in positive characteristic, and no counterexample is currently known. \end{rem}
\begin{proof} 
    By Remark \ref{rem:Frobenius liftable singularities} (1) and (2), the pair $(X, D)$ is normal crossing in codimension $\leq 2$ in $X$. 
    Thus, by Lemma \ref{lem:exact sequence}, we have the short exact sequence
    \[
    0\to \Omega^1_X\to \Omega^{[1]}_X(\log D)\to \nu_*\mathcal{O}_{D'}\to 0
    \]
    for the normalization $\nu\colon D'\to D$.
    {In turn, we get a} long exact sequence
    \begin{multline*}
        0\to H^0(X, \Omega^1_X)\to H^0(X, \Omega^{[1]}_X(\log D))\to  k\cong H^0(D',\sO_{D'})\\
        \xrightarrow{\delta_X} H^1(X,\Omega^{1}_X)\to H^1(X,\Omega^{[1]}_X(\log D))=0,
    \end{multline*}
    where the last vanishing follows from log Bott vanishing (Proposition \ref{prop:log Bott}).
    From the above exact sequence, we can deduce the assertions.
\end{proof}

\begin{lem}\label{lem:general line}
         Let $k$ be an algebraically closed field, and
        let $D$ be a prime divisor on $X\coloneqq \mathbb{P}^n_k$.
        If $\deg(D)\geq 2$, then $\dim_{k}H^0(L, \Omega_X^{[1]}(\log D)|_L)\geq 1$ for a general line $L$.
\end{lem}
    \begin{proof}
  Since $L$ is general, we may assume that $L$ intersects $D$ at smooth points only, and thus that $\Omega_X^{[1]}(\log D)|_L$ is locally free. 
  Hence $\Omega_X^{[1]}(\log D)|_L\cong \oplus_{i=1}^n \mathcal{O}_L(a_i)$ for some integers $a_i$.  
Then we have an equality
\begin{align*}
    \sum_ia_i&=\text{ deg }( \Omega_X^{[1]}(\mathrm{log} D)|_L)=\deg(\omega_X(D)|_L)=\deg_X(D)-(n+1).
\end{align*}
If $\deg(D) \geq 2$, then at least one of the $a_i$ must be non-negative, and we have
$\dim_k H^0(L, \Omega_X^{[1]}(\log D)|_L) \geq 1$.
\end{proof}

\begin{lem}\label{small codim}
    Let $k$ be an algebraically closed field.
    \begin{enumerate}
        \item[\textup{(1)}] Let $Z_1$ be a proper closed subset of $\mathbb{P}^n_k$ and let $W_1 \subseteq\Gr(2,n+1)$ be the closed subset consisting of lines in $\mathbb{P}^n_k$ contained in $Z_1$. Then $\dim W_1\leq 2n-4$.
        \item[\textup{(2)}] Let $Z_2$ be a closed subset of $\mathbb{P}^n_k$ satisfying $\dim Z_2\leq n-3$ and let $W_2 \subseteq \Gr(2,n+1)$ be the closed subset consisting of lines in $\mathbb{P}^n_k$ intersecting $Z_2$. Then $\dim W_2\leq 2n-4$.
    \end{enumerate}
\end{lem}
\noindent Recall that $\Gr(2,n+1)$ is the Grassmanian variety of lines in $\mathbb{P}^n_k$ and $\dim \Gr(2,n+1) = 2n-2$.
\begin{proof} 
We first prove (1). Let $D_1 \subseteq W_1\times Z_1$ be the closed subset consisting of pairs \[(L, x)\in W_1\times Z_1\] such that $x\in L$. Note that, by definition, $L \subseteq Z_1$. We have two projections $p_1\colon D_1\to W_1$ and $p_2\colon D_1\to Z_1$. 
        For $z_1\in Z_1$, we have that $p_2^{-1}(z_1) \subseteq W_1$ is a closed subvariety of the variety $\mathbb{P}^{n-1}$ of the lines in $\mathbb{P}^n_k$ passing through $z_1$ and contained in $Z_1$.
        Since $Z_1\neq \mathbb{P}^n_k$, 
        the subvariety $p_2^{-1}(z_1)$ does not coincide with $\mathbb{P}^{n-1}$, and thus $\dim p_2^{-1}(z_1)\leq n-2$. Now, we obtain
    \[
    \dim D_1\leq \dim Z_1+ \dim p_2^{-1}(z_1)\leq n-1+n-2=2n-3.
    \] 
Finally, each fiber of  $p_1$ is $\mathbb{P}^1$, and so we conclude that $\dim W_1=  \dim D_1-1\leq 2n-4$.

We next prove (2). Let $D_2 \subseteq  W_2\times Z_2$ be the closed subset consisting of pairs \[(L, x)\in W_2\times Z_2\] such that $x\in L$. We have two projections $p_1\colon D_2\to W_2$ and $p_2\colon D_2\to Z_2$.
For $z_2\in Z_2$, note that $p_2^{-1}(z_2)$ is the variety $\mathbb{P}^{n-1}$ of the lines in $\mathbb{P}^n_k$ passing through $z_2$. Thus, $\dim p_2^{-1}(z_2)= n-1$, and we have
    \[
    \dim D_2\leq \dim Z_2+ \dim p_2^{-1}(z_2)\leq n-3+n-1=2n-4.
    \] 
    Since $p_1$ is surjective, we conclude that $\dim W_2\leq \dim D_2\leq 2n-4$.
\end{proof}

\begin{thm}\label{thm:key}
   Let $k$ be an algebraically closed field of characteristic $p>0$ and let $D$ be a prime divisor on $X\coloneqq\mathbb{P}^n_k$. Suppose that $(X, D)$ is $F$-liftable. 
   Then $\deg(D)=1$.
\end{thm}
\begin{proof}
    Since $H^1(X, \Omega^1_X) \neq 0$ and $D$ is an ample prime divisor,  Theorem \ref{thm:global section of F-lift pair} implies 
    \[H^0(X, \Omega^{[1]}_{X}(\log D))=H^0(X,\Omega^1_X)=0.
    \]
    Suppose by contradiction that $\deg(D)\geq 2$.
    Under this assumption, we will show the following claim: 
    \begin{equation} \label{eq:claim-key}
        \dim_{k}H^0(X, \Omega^{[1]}_{X}(\log D))\neq 0
    \end{equation}
    contradicting the above vanishing. 
    
    Let $ V\subseteq X$ be the normal crossing locus of $(X,D)$.
    Set $Z\coloneqq X\setminus V$. By Remark \ref{rem:Frobenius liftable singularities}, we have $\mathrm{codim}_X(Z)\geq 3$.
    Next, by (\ref{eq:xi-Frobenius-lift}), we have a map
    \[
    F^*\Omega^1_V(\log D)\xrightarrow{\xi} \Omega^1_V(\log D).
    \]
     Let $U\subseteq V$ be the open dense subset where $\xi$ is an isomorphism.

   Let  $\mathcal{M}:=\Gr(2,n+1)$ be the Grassmanian variety of lines in $\mathbb{P}^n_k$, and let \[\mathcal{U}'\coloneqq\{(x,L)\in X\times \mathcal{M}\mid x\in L\}\] be 
   the universal family over $\mathcal M$. In particular, we have a natural diagram
   \begin{center}
   \begin{tikzcd}
    \mathcal{U}' \arrow{r}{\phi'} \arrow{d}{\pi'} & X \\
    \mathcal{M} & 
   \end{tikzcd}.
   \end{center}
   
   Let $\mathcal{M}_0\subseteq \mathcal{M}$ be the space of lines $L \subseteq X$ such that $L\cap U\neq \emptyset$ and $L\cap Z= \emptyset$; namely, $L$ intersects the locus where $\xi$ is an isomorphism and $L$ is entirely contained in the normal crossing locus of $(X,D)$.
  By applying Lemma \ref{small codim} to $(Z_1, Z_2)=(X\setminus U, Z),$ we get that $\mathrm{codim}_{\mathcal{M}} (\mathcal{M}\setminus\mathcal{M}_0)\geq 2$. 
  
  Let $\mathcal{U}_0=\mathcal{U}'\times_\mathcal{M} \mathcal{M}_0$
  and $\mathcal{U}\coloneqq \mathcal{U}'\times_X V$. Namely, $\mathcal{U}$ parametrizes pairs $(x,L) \in X \times \mathcal{M}$ such that $x \in L$ and $x$ is contained in the normal crossing locus $V$, while $\mathcal{U}_0$ parametrizes pairs $(x,L)$ such that $x \in L$ and $L \in \mathcal M_0$. These are both open subschemes of $\mathcal{U}'$, and we have a natural inclusion $j\colon \mathcal{U}_0\hookrightarrow \mathcal{U}$.
  
  We can summarize the above discussion with the following commutative diagram:
  \begin{center}
\begin{tikzcd}
& \mathcal{U}\arrow[d, hookrightarrow]\arrow[r, "\phi"]\arrow[dd,bend left,near start,"\pi"] & V\arrow[d, hookrightarrow]\\
\mathcal{U}_0 \arrow[r,hookrightarrow]\arrow[d,"\pi'|_{{\mathcal{U}_0}}"']\arrow[ru,hookrightarrow,"j"] &\mathcal{U}' \arrow[r, "\phi'"]  \arrow[d, "\pi'"'] & X\\ 
\mathcal{M}_0 \arrow[r,hookrightarrow] & \mathcal{M}.  &
\end{tikzcd}
\end{center}
  
\noindent Recall that $\pi'$ is a $\mathbb{P}^1$-bundle and $\phi'$ is a $\mathbb{P}^{n-1}$-bundle.
Thus, $\phi$ is a $\mathbb{P}^{n-1}$-bundle and $\pi'|_{\mathcal{U}_0}\colon \mathcal{U}_0\to \mathcal{M}_0$ a $\mathbb{P}^{1}$-bundle. Since $\mathrm{codim}_{\mathcal{M}} (\mathcal{M}\setminus\mathcal{M}_0)\geq 2$ and $\pi'$ is a $\mathbb{P}^{1}$-bundle, we have $\mathrm{codim}_{\mathcal{U}'} (\mathcal{U}'\setminus\mathcal{U}_0)\geq 2$, and, in particular, $\mathrm{codim}_{\mathcal{U}} (\mathcal{U}\setminus\mathcal{U}_0)\geq 2$.\\

Recall that our goal is to show that $H^0(X, \Omega^{[1]}_X(\log D))\neq 0$. 
Since $\mathrm{codim}_X Z\geq 3$, we have 
\[
H^0(X, \Omega^{[1]}_X(\log D))=H^0(V, \Omega^1_V(\log D)).
\]
Now consider the \'etale sheaf $\Omega^1_V(\log D)^{\xi}$ on $V$ from Lemma \ref{lemma:magic-cover}.  {To conclude the proof it suffices to prove that $\dim_{k}H^0(V, \Omega^1_V(\log D)^{\xi})\neq 0$.}
To see this, it is enough to construct a nonzero constant sheaf contained in $\Omega^1_V(\log D)^{\xi}$. As $\mathbb{P}^n_k$ is simply connected and $\mathrm{codim}_X Z\geq 3$, we have that $V$ is simply connected by Zariski--Nagata purity. Thus, any locally constant \'etale sheaf on $V$ is constant, and so it is enough to show that $\Omega^1_V(\log D)^{\xi}$ contains a nonzero locally constant sheaf.

Define an \'etale subsheaf $\mathcal{G}$ of $\phi^*\Omega^1_V(\log D)^{\xi}|_{\mathcal{U}_0}$ by $\mathcal{G}=(\pi^*\pi_*\phi^*\Omega^1_V(\log D)^{\xi})|_{\mathcal{U}_0}$. Informally speaking, $\mathcal G$ should be the maximal subsheaf of $\phi^*\Omega^1_V(\log D)^{\xi}|_{\mathcal U_0}$ which is constant on each fiber of $\pi \colon \mathcal U_0 \to \mathcal M_0$. Note that 
\begin{align*}
j_* (\phi^*\Omega^1_V(\log D)^{\xi}|_{\mathcal{U}_0}) &=(j_*\phi^*\Omega^1_V(\log D)|_{\mathcal{U}_0})^{\xi} \\
&=(\phi^*\Omega^1_V(\log D))^{\xi} \qquad \text{ (as codim}_{\mathcal{U}} (\mathcal{U}\setminus\mathcal{U}_0)\geq 2)\\
&=\phi^*\Omega^1_V(\log D)^{\xi}.
\end{align*}

Since all fibers of $\phi$ are isomorphic to $\mathbb{P}^{n-1}$, we have \[
\phi_*j_* (\phi^*\Omega^1_V(\log D)^{\xi}|_{\mathcal{U}_0})=\phi_*\phi^*\Omega^1_V(\log D)^{\xi}=\Omega^1_V(\log D)^{\xi}.
\]
Thus, $\phi_*j_*\mathcal{G}$ is an {\'e}tale subsheaf of $\Omega^1_V(\log D)^{\xi}$. Our aim is to show that $\phi_*j_*\mathcal{G}$ is nonzero locally constant. The following claim is needed:

\begin{cl} 
$\mathcal{G}$ is a locally constant \'etale sheaf of rank $r > 0$. 
\end{cl}

Assuming the Claim, we prove the theorem. Since $\codim_{\mathcal{U}} (\mathcal{U}\setminus\mathcal{U}_0)\geq 2$, we have that $j_* \mathcal{G}$ is also locally constant of rank $r$ by Zariski-Nagata purity. Thus, for any fiber $T$ of $\phi$, we have that $j_* \mathcal{G}|_T$ is locally constant of rank $r$.
Since $T\cong \mathbb{P}^{n-1}$ is simply connected, 
$j_* \mathcal{G}|_T$ must be constant.
Thus, $\dim_{\mathbb{F}_p}H^0(T,j_* \mathcal{G}|_T)=r$ for all fibers $T$ of $\phi$. By the proper base change theorem, the stalk of $\phi_*j_*\mathcal{G}$ at every closed point has dimension $r$, and thus $\phi_*j_*\mathcal{G}$ is locally constant of rank $r > 0$. \\

Now we prove the Claim. It suffices to show that there is $r>0$ such that for all $m\in\mathcal{M}_0$, the restriction $\mathcal{G}|_{\pi^{-1}(m)}$ is a constant sheaf of rank $r$. Equivalently, we need to show that there is $r>0$ such that $\dim_{\mathbb{F}_p}(\pi_* \phi^*\Omega^1_V(\log D)^{\xi}|_{\mathcal{U}_0})_m=r$ for all $m\in \mathcal{M}_0$.
By the proper base change theorem, we have
\begin{align*}
    \dim_{\mathbb{F}_p}(\pi_* \phi^*\Omega^1_V(\log D)^{\xi}|_{\mathcal{U}_0})_m &=\dim_{\mathbb{F}_p} H^0(\pi^{-1}(m), \phi^*\Omega^1_V(\log D)^{\xi}|_{\pi^{-1}(m)})\\
    &=\dim_{\mathbb{F}_p} H^0(\pi^{-1}(m),\phi^*\Omega^1_V(\log D)|_{\pi^{-1}(m)})^{\xi}\\
    &=h^0(\pi^{-1}(m), \phi^*\Omega^1_V(\log D)|_{\pi^{-1}(m)}),
\end{align*}
where the last equality follows from Lemma \ref{lem:p-linear}. Since $\Omega^1_V(\log D)$ is locally free, the sheaf $\phi^*\Omega^1_V(\log D)|_{\mathcal{U}_0}$ is flat over $\mathcal{M}_0$. By semi-continuity theorem, the last term in the above equalities is upper semicontinuous on $\mathcal{M}_0$. By \cite[Lemma 6.2.3]{AWZ}, the first term is lower semicontinuous on $\mathcal{M}_0$. Therefore, it is constant of dimension $r$ for all $m\in \mathcal{M}_0$, and by Lemma \ref{lem:general line}, we have $r>0$.
This proves the claim, and we conclude the theorem.
 \end{proof}

\subsection{Proof of the main theorem and applications}
We are ready to provide to the proofs of the main theorem and its applications.

\begin{proof}[Proof of Theorem \ref{Introthm:main}]
{For any irreducible component $D'$ of $D$, we have that $\deg(D')=1$ by Theorem \ref{thm:key} since $(X,D')$ is $F$-liftable.} Thus, we can write $D = \sum_{i=1}^N H_i$ for some distinct hyperplanes $H_i = \{f_i = 0\}$. It then suffices to show that $f_i$ are linearly independent, since in that case $f_i$ form part of a basis of $(k^{n+1})^*$ up to a linear automorphism of $\mathbb{P}^n_k$.
Suppose, for contradiction, that $f_i$ are not linearly independent.  Then we can take $0<m<N$ such that $f_1, f_2, \ldots, f_m$ are linearly independent, but $f_1, f_2, \ldots, f_{m+1}$ are linearly dependent. As $(\mathbb{P}^n_k, D)$ is $F$-liftable, it is globally $F$-split (Remark \ref{rem:Frobenius liftable singularities} (1)), and $-(K_{\mathbb{P}^n_k}+D)$ is effective (cf.~\cite[Theorem 4.3 (2)]{SS10}).
This shows $N\leq n+1$, and $m\leq n$. Thus, we have
\[
\emptyset\neq W \coloneqq \bigcap_{i=1}^m H_i \subset H_{m+1}.
\]
Let $E$ be the exceptional divisor of the blow-up of $\mathbb{P}^n_k$ along $W$. Then, by \cite[Lemma 2.29]{KM98}, we have the following inequality of discrepancies:
\[
a(E, \mathbb{P}^n_k, D) \leq a\left(E, \mathbb{P}^n_k, \sum_{i=1}^{m+1} H_i\right) = -2.
\]
This contradicts the assumption that $(\mathbb{P}^n_k, D)$ is lc (Remark \ref{rem:Frobenius liftable singularities}(2)).
\end{proof}

\begin{thm}\label{thm:toric image in char p}
    Let $X$ be a smooth projective variety of Picard rank 1 over an algebraically closed field $k$ of characteristic $p > 0$, and let $D$ be a reduced divisor on $X$. Suppose there exists a finite surjective morphism $f\colon Y \to X$ of degree prime to $p$ such that $(Y, D_Y)$ is a toric pair. Then $X \cong \mathbb{P}^n_k$, and $D$ is a toric divisor with respect to the standard toric structure on $\mathbb{P}^n_k$, up to an automorphism of $\mathbb{P}^n_k$.
    Here, $D_Y$ denotes the sum of the components of $f^{-1}(D)_{\mathrm{red}}$ along which the ramification index of $f$ is prime to $p$.
\end{thm}
\begin{proof}
    By Theorem \ref{thm:Occhetta-Wisniewski}, $X \cong \mathbb{P}^n_k$.
    Since $(Y,D_Y)$ is  toric, it is $F$-liftable (cf.~\cite[Remark 2.7 (3)]{Kaw4}), and
    $(X,D)$ is $F$-liftable by \cite[Theorem 3.11]{Kawakami-Takamatsu}.
    Now, by Theorem \ref{Introthm:main}, we conclude.
\end{proof}

\begin{proof}[Proof of Theorem \ref{Introthm:main2}]
   We take a model $f_A\colon (Y_A,f_A^{-1}(D_A))\to (X_A,D_A)$ of $f\colon (Y,f^{-1}(D))\to (X,D)$ over a finitely generated $\mathbb{Z}$-subalgebra $A$ of $k$.
   By enlarging $A$, for every closed point $s\in \Spec A$, we may assume that $\deg(f_s)$ is prime to $p(s)=\mathrm{char}(k(s))$ and ramification degree of $f_{s}$ along every component $f_s^{-1}(D_s)$ is prime to $p(s)$.
   As proven in Theorem \ref{thm:toric image in char p}, 
   $(X_s,D_s)$ is $F$-liftable, $X_s\cong \mathbb{P}_{k(s)}^n$, and $\deg_{X_s}(D'_s)=1$ for every component $D'_s$ of $D_s$.
   Therefore, $(X,D)$ is lc, $X\cong \mathbb{P}_{k}^n$, and $\deg_{X}(D')=1$ for every component $D'$ of $D$. 
   Now, as in the proof of Theorem \ref{Introthm:main}, we can conclude that all defining equations of irreducible components of $D$ are linearly independent.
\end{proof} 

\providecommand{\bysame}{\leavevmode\hbox to3em{\hrulefill}\thinspace}
\providecommand{\MR}{\relax\ifhmode\unskip\space\fi MR }
\providecommand{\MRhref}[2]{%
  \href{http://www.ams.org/mathscinet-getitem?mr=#1}{#2}
}
\providecommand{\href}[2]{#2}


\end{document}